\documentclass{article}
\usepackage[latin1]{inputenc}
\usepackage{amsmath}
\usepackage{amsthm}

\usepackage{color}
\usepackage{amsfonts}
\usepackage{amssymb}

\begin{document}
 
\newtheorem{thm}{Theorem}[section]
\newtheorem*{resthm}{Theorem}
\newtheorem{cor}[thm]{Corollary}
\newtheorem{lem}[thm]{Lemma}
\newtheorem{claim}[thm]{Claim}
\newtheorem{prop}[thm]{Proposition}
\theoremstyle{definition}
\newtheorem{rem}[thm]{Remark}
\newtheorem{defn}[thm]{Definition}
\newcounter{qu}
\newcounter{cs}
\theoremstyle{remark}
\newtheorem*{remun}{Remark}
\newtheorem{question}[qu]{Question}
\newtheorem{case}[cs]{Case}

\renewcommand{\emptyset}{\varnothing}
\newcommand{\bound}{\mathcal C}

\newcommand{\restrict}[2]{{{#1}\! \!   \restriction_{\! #2}}}

\newcommand{\sm}{\smallsetminus}
\newcommand{\R}{\mathbb {R}}
\newcommand{\T}{\mathbb {T}}

\newcommand{\Q}{\mathbb {Q}}
\newcommand{\Z}{\mathbb {Z}}
\newcommand{\N}{\mathbb {N}}
\newcommand{\Lf}{\mathcal {L}}
\newcommand{\inv}{^{-1}}
\newcommand{\locUnst}[2][\epsilon]{W^u_{#1}(#2)}
\newcommand{\locUnstBd}[2][\epsilon]{W^u_{\bound, #1}(#2)}
\newcommand{\unst}[1]{W^u(#1)}
\newcommand{\stab}[1]{W^s(#1)}
\newcommand{\locStab}[2][\epsilon]{W^s_{#1}(#2)}
\newcommand{\supp}{{\mathrm{supp}~}}
\newcommand{\nw}[1]{\mathrm{NW}{(#1)}}
\newcommand{\dist}{\mathrm{dist}\ }
\newcommand{\len}{\mathrm{len}\ }
\newcommand{\Id}{\mathrm{Id}\ }
\newcommand{\tr}{\mathrm{tr}\ }
\newcommand{\NW}{{\mathrm{NW}}}
\newcommand{\Per}{{\mathrm{Per}}}

\newcommand{\cmt}[1]{{\color{red}{{#1}}}}



\title{Constraints On Dynamics Preserving Certain Hyperbolic Sets}

\author{Aaron W. Brown}


\maketitle

\begin{abstract}
We establish two results under which the topology of a compact hyperbolic set 
constrains ambient dynamics.  First if $\Lambda $ is a transitive, codimension-1, expanding attractor for some diffeomorphism, then  $\Lambda$ is a union of transitive, codimension-1 attractors (or contracting repellers) for any  diffeomorphism such that $\Lambda$ is hyperbolic.  Secondly, if $\Lambda $ is a locally maximal nonwandering  set for a surface diffeomorphism,  then  $\Lambda$ is   locally maximal for any diffeomorphism for which $\Lambda$ is hyperbolic.  
\end{abstract}

\section{Introduction}
Consider the following problem: given a subset $\Lambda$ of a manifold $M$,  describe the group of diffeomorphisms of $M$ preserving $\Lambda$ in terms of the topology of $\Lambda$.    Clearly in many examples, the topology of $\Lambda$ is too tame for it to impose interesting constraints (for example when $\Lambda$ is finite or  $\Lambda = M$).  However  through additionally dynamical hypotheses, we may refine the above into a number of more tractable problems.  For example, given a diffeomorphism $f\colon M\to M$ and a hyperbolic set  $\Lambda$ for $f$, can one describe the group of diffeomorphisms commuting with $f$ and preserving  $\Lambda$?   
Questions of this nature has been addressed extensively in the case when $\Lambda = M$; in particular, it is conjectured that that all irreducible Anosov actions of $\Bbb{Z}^k$ or $\Bbb{R}^k$, $k\ge 2$, are smoothly conjugate to algebraic actions (see for example \cite{MR1632177}).  

Alternatively we could pose the following: given $\Lambda\subset M$, can one  describe the set of all diffeomorphisms of $M$ for which $\Lambda$ is a hyperbolic set?
In considering this question, one might ask, for example, if $\Lambda$ is a hyperbolic attractor for some diffeomorphism $f\colon M\to M$, must  $\Lambda$  be an attractor for all diffeomorphisms $g\colon M\to M$ such that $\Lambda$ is a hyperbolic set?  In \cite{Fisher:2006p1846}  a natural counterexample is constructed.  However, in the same paper the following  result is proved, demonstrating strong constraints on the set of diffeomorphisms preserving  hyperbolic attractors in surfaces.  
\begin{resthm}[Fisher \cite{Fisher:2006p1846}]\label{thm:Fisher}
If $M$ is a compact surface and $\Lambda$ is a nontrivial topologically mixing hyperbolic attractor for a diffeomorphism $f$ of $M$, and $\Lambda$ is hyperbolic for a diffeomorphism $g$ of $M$, then $\Lambda$ is either a nontrivial topologically mixing hyperbolic attractor or a nontrivial topologically mixing hyperbolic repeller for $g$. 
\end{resthm}

In this paper we present two Theorems that generalize Theorem \ref{thm:Fisher}.
  The first generalizes the theorem to  expanding  attractors of codimension 1 in  manifolds of arbitrary dimension.  Recall that a hyperbolic attractor $\Lambda\subset M$ is \emph{expanding} if $\dim(\Lambda) =  \restrict {\dim E^u}  \Lambda$, and  is said to be of \emph{codimension 1} if $\restrict {\dim E^u}\Lambda = \dim M -1$.  
\begin{thm}\label{thm:main}
Let $f\colon   M\to M$ be a 
diffeomorphism  
and let $\Lambda$ be a compact, topologically mixing,   expanding  hyperbolic  attractor of codimension 1. 
Suppose $g\colon   M\to M$ is a second diffeomorphism 
such that $\Lambda$ is hyperbolic for  $g$.
Then $\restrict g {\Lambda}$ has a codimension-1 hyperbolic splitting  and $\Lambda$ is a topologically mixing expanding attractor  (or contracting  repeller) of codimension 1  for $g$.  \end{thm}

Dropping the assumption of topological mixing, a straightforward generalization of the above is the following.

\begin{thm}\label{thm:main2}
Let $f\colon   M\to M$ be a
diffeomorphism and let $\Lambda$ be a compact, transitive,   expanding hyperbolic  attractor of codimension 1. 
Suppose $g\colon   M\to M$ is a second diffeomorphism 
such that $\Lambda$ is hyperbolic for $g$.   
Then $\Lambda$ decomposes into a finite number of  compact sets $\{\Lambda_j\}$ such that $\restrict g {\Lambda_j}$ has a codimension-1 hyperbolic splitting, and  $\Lambda_j$ is a transitive   expanding   attractor  (or contracting  repeller) of codimension-1  for $g$.  
\end{thm}

\begin{remun}
 Portions of Theorems \ref{thm:main} and \ref{thm:main2} hold in the slightly trivial case that $\Lambda$ is non-expanding.  Indeed by Theorem 4.3 of \cite{MR0006493}, if $\dim(\Lambda) = \dim M$ then $\Lambda$ has nonempty interior, which according to Theorem 1 of \cite{MR2199438} implies that $\Lambda = M$.  
Then $\Lambda = M$ will trivially be a hyperbolic attractor for any diffeomorphism $g\colon M\to M$.  However, we don't expect the constraint  on the dimensions of the hyperbolic splitting in the conclusion of Theorem \ref{thm:main}  to hold. Indeed
 by the Franks-Newhouse Theorem for codimension-1 Anosov diffeomorphisms (\cite{MR0271990}, \cite{MR0277004}),  this would imply $M = \T^{n}$, hence we could find hyperbolic linear maps on $M$ with arbitrary dimensional splittings. \end{remun}

Secondly, we generalize Theorem \ref{thm:Fisher} to locally maximal subsets of surfaces with nontrivial recurrence.  
We will  call a set \emph{nonwandering} if $ \Lambda\subset \NW(f)$. 
\begin{thm}\label{thm:LM}
Let $f\colon S\to S$ be a 
diffeomorphism of a surface $S$, and let $\Lambda$ be a nonwandering, locally maximal, compact hyperbolic set for $f$.  Assume $g\colon S\to S$ is a second diffeomorphism such that $ \Lambda$ is hyperbolic for $g$.  
Then $\Lambda$ is locally maximal for $g$.
\end{thm}

As a corollary to Claim \ref{clm:cantor} and Proposition \ref{lem:cont} used in the proof of Theorem \ref{thm:LM} we obtain that the stable and unstable manifolds for $g$ and $f$ are essentially algined in the case that $\Lambda$ is totally disconnected. 
\begin{cor} \label{cor:algined}
Let $\Lambda$, $f$, and $g$ be as in Theorem \ref{thm:LM} and assume that $\Lambda$ is totally disconnected.  Then  we may find some $\epsilon, \epsilon'>0$ so that for all $x\in \Lambda$ \[W^u_{f,\epsilon}(x) \cap \Lambda \subset W^\sigma_{g,\epsilon'}(x) \cap \Lambda\] and \[W^s_{f,\epsilon}(x) \cap \Lambda \subset W^{\sigma'}_{g,\epsilon'}(x) \cap \Lambda\]  for $\sigma, \sigma' \in \{u,s\}$ (depending on $x$), where $W^\sigma_{g,\epsilon}(x)$ and $W^\sigma_{f,\epsilon}(x)$ denote the local $\sigma$-manifolds for the dynamics of $g$ and $f$ respectively.
\end{cor}

\begin{remun}By Remark \ref{lem:perfecttrans} (below), one could replace the hypothesis  that $\Lambda$ is nonwandering   in Theorem \ref{thm:LM} with the hypothesis that $\Lambda$ is transitive and perfect.
\end{remun}

We leave the reader with the following questions that would extend our results and address some of the more general problems outlined above.  Considering the counterexample to extending Theorem \ref{thm:Fisher} to arbitrary hyperbolic attractors, as constructed in \cite{Fisher:2006p1846}, 
we pose the following questions. 
\begin{question} Let $\Lambda$  be a compact, topologically mixing,  hyperbolic attractor for a diffeomorphism $f\colon M\to M$,  and suppose that $\Lambda \not \subset S$ for any embedded submanifold $S\subset M$ with $\dim S <\dim M$.  If a second diffeomorphism $g\colon M\to M$ is such that $\Lambda$ is hyperbolic for $g$, is $\Lambda$ a topologically mixing attractor (or repeller) for $g$?\end{question}

\begin{question} If $\Lambda$ is a compact hyperbolic set for a diffeomorphism $g\colon M\to M$, in what ways does the topology of $\Lambda$ constrain the dimensions of the hyperbolic splitting? In particular, if $\Lambda$ is a mixing expanding attractor for $f\colon M\to M$, and $\Lambda$ is hyperbolic for a diffeomorphism $g\colon M\to M$, can ${\dim E^u_g (x)} <{\dim E^u_f (x)} $ for some $x\in \Lambda$?
\end{question}
The following question is related to Problem 1.4 posed by Fisher in \cite{Fisher:2006p1846}; it would generalize Theorem \ref{thm:LM} to arbitrary dimensions.
  
\begin{question} Let $\Lambda$  be a nonwandering, locally maximal, compact  hyperbolic set for a diffeomorphism $f\colon M\to M$.   If a second diffeomorphism $g\colon M\to M$ is such that $\Lambda$ is a hyperbolic set, is $\Lambda$ locally maximal for $g$? \end{question}
\begin{question}If the answer to Question 3 is negative, does the result hold under the additional assumption that $\Lambda$ is totally disconnected? \end{question}
Finally, we pose the following question, which  would require the analysis of locally maximal hyperbolic sets whose local product structure is non-uniform.  
\begin{question} Can the assumption that $\Lambda$ is nonwandering be dropped in Theorem \ref{thm:LM}?  In particular, does the conclusion remain true if $\Lambda$ contains isolated points or if $\Lambda$ is not transitive?
\end{question}

\section{Preliminaries.}
The basic properties of uniformly hyperbolic dynamics over compact sets are presented in the literature.  See for example \cite{MR0271991}, \cite{Katok:1995p8357}, and \cite{Smale:DDS}.   We briefly outline the main results needed here.  
Let $M$ be a smooth manifold endowed with a Riemannian metric. Given a diffeomorphism $f\colon M\to M$ we say that $\Lambda$ is an \emph{invariant set} if $f(\Lambda)\subset \Lambda$.  
We call a set $\Lambda\subset M$   \emph{hyperbolic} for $f$ if it is invariant and if there exist constants $\kappa<1$ and $C>0$, 
  and a continuous $Df$-invariant splitting of the tangent bundle $T_x M = E^s(x)\oplus E^u(x)$ over $\Lambda$   
 so that for all points $x\in \Lambda$ and any $n \in \N$
\begin{align*}
   \|Df^n_x v\|\le C \kappa^n \|v\|,  \quad &\mathrm{ for } \ v\in E^s(x)\\
 \|Df^{-n}_x v\|\le C \kappa^{n}\|v\|,  \quad &\mathrm{ for }\   v\in E^u(x).
\end{align*}
If $\Lambda$ is a compact hyperbolic set for $f$, it is possible to  endow $M$ with a smooth metric such that we may take $C = 1$ above.  Such a metric is said to  be \emph{adapted} to the dynamics of $f$ on $\Lambda$.  Let us choose such a metric.  

If $\Lambda$ is a compact hyperbolic  set, then there exists an $\epsilon>0$ so that for 
each point $x\in \Lambda$, the sets 
\begin{align*}
W^s_\epsilon (x) &:= \{y\in M\mid d(f^n(x),f^n(y))< \epsilon, \  \mathrm{for\ all} \  n\ge 0\} \\ W^u_\epsilon (x) &:= \{y\in M\mid d(f^n(x),f^n(y))< \epsilon, \ \mathrm{for\ all} \  n\le 0\}
\end{align*}
are embedded open disks, called the \emph{local stable} and \emph{unstable} manifolds.  If $f\colon M\to M$ is a $C^k$ diffeomorphism ($k\ge 1$) then  for $\sigma = \{s, u\}$   each $W^\sigma_\epsilon(x)$ is $C^k$ embedded, and the family $\{W^\sigma_\epsilon(x)\}_{x\in \Lambda}$ varies continuously with $x$.  We have $T_x W^\sigma_\epsilon (x) = E^\sigma(x)$.

Furthermore, if $d$ is the distance function on $M$ induced by  an adapted metric, we can find $\lambda<1<\mu$ with the property that for  all $x\in \Lambda, y\in \locStab[\epsilon]x, z\in \locUnst[\epsilon]x$ and $n\ge 0$ we have
\begin{align}
d(f^n(x), f^n(y))& \le \lambda^n d(x,y)\label{eqn:asymS}\\
d(f^{-n}(x), f^{-n}(z))& \le \mu^{-n} d(x,z).\label{eqn:asymU}
\end{align}

Note that (\ref{eqn:asymS}) and (\ref{eqn:asymU})  imply that  $ f(W_{\epsilon}^s(f\inv(x)))\subset W_{\epsilon}^s(x)$ and $W_{\epsilon}^u(x) \subset f(W_{\epsilon}^u(f\inv(x)))$. 
We define \begin{align*}\unst x &:= \bigcup _{n \in \N} f^n(W^u_{\epsilon}(f^{-n}(x))) \\ \stab x &:= \bigcup _{n \in \N} f^{-n}(W^s_{\epsilon}(f^{n}(x))) \end{align*} called the \emph{global} stable and unstable manifolds.  We note that $\unst x \cong \R^{\dim E^u(x)}$ and $\stab x \cong \R^{\dim E^s(x)}$ are $C^k$ injectively immersed submanifolds.  

For the remainder of this article we will always use a metric adapted to our dynamics and denote by $d$ the induced distance function on $M$.  The constants $\lambda<1<\mu$ will always be as in  (\ref{eqn:asymS}) and (\ref{eqn:asymU}).  We shall call the $\epsilon$ satisfying  the above properties the \emph{radius of the local stable and unstable manifolds} of $\Lambda$, usually denoted by $\epsilon_0$.  We will denote $B(x,\delta):=\{y\in M\mid d(x,y)< \delta\}.$


A point $x$ in a topological space $X$ is said to be \emph{nonwandering}   for a homeomorphism $f\colon X \to X$ if for 
any open $U\ni x$, there is some $n> 0$ such that $f^n(U)\cap U \neq \emptyset$; otherwise it is called \emph{wandering}.  
We denote by $\NW(f)$ the set of all nonwandering points for $f$.  
{We will call an invariant set $\Lambda$ \emph{nonwandering} if $\Lambda\subset \NW(f)$.  

An invariant set $\Lambda$ is called \emph{topologically transitive} if it contains a dense orbit.  Alternatively, a compact invariant subset $\Lambda\subset M$ is topologically transitive if for all pairs of nonempty open sets $U, V\subset \Lambda$, there is some $n$ such that $f^n(U)\cap V \neq \emptyset$.  Finally, an invariant set $\Lambda$ is called \emph{topologically  mixing} if for all pairs of nonempty  open sets $U, V \subset \Lambda$, there is some $N$ such that $f^n(U)\cap V \neq \emptyset$ for all $n\ge N$.  
We note that for arbitrary compact transitive sets $\Lambda$, even those that are hyperbolic and locally maximal we may have $\NW(\restrict f \Lambda) \neq \Lambda$ or even $\Lambda \not\subset \NW(\Lambda)$ (for example, consider the closure of the orbit of a transverse heteroclinic intersection).  However we have the following basic observation.   
\begin{rem}\label{lem:perfecttrans}
Let $X\subset M$ be a compact transitive set for $f$.  If $X$ contains no isolated points then $\NW(\restrict f \Lambda) = \Lambda$.  
\end{rem}
Indeed, let $x$ have a dense orbit.  For any $y\in X $ and open $y\in U\subset X$ we can find $n, k \in \Z$ so that $f^n(x)\subset U\sm \{y\}$ and $f^k(x) \subset U\sm \{y, f^n(x)\}$.   But then $f^{|k-n|} (U) \cap U \neq \emptyset$.

A compact, hyperbolic, invariant set $\Lambda$ is a \emph{hyperbolic  attractor} if there is some open neighborhood $ \Lambda\subset V$ such that $\bigcap_{n\in \N}f^n(V) = \Lambda$.  Alternatively, if $\Lambda$ is compact hyperbolic set, then it is a hyperbolic attractor if and only if $\unst x \subset \Lambda$ for all $x\in \Lambda$.  If $\Lambda$ is a topologically  mixing attractor, then for each $x\in \Lambda$, $\unst x$ is dense in $\Lambda$.  For a hyperbolic attractor $\Lambda$, the \emph{basin} of $\Lambda$ is the set $\bigcup_{y\in \Lambda}\stab y$.

We recall from the introduction that hyperbolic attractor $\Lambda$ is called \emph{expanding} if the topological dimension of $\Lambda$ equals the dimension of the unstable manifolds.  (For an introduction to topological dimension see \cite{MR0006493}.)  By a \emph{contracting repeller} we mean an expanding attractor for $f\inv$.  Additionally recall that  if $M$ is an $(n+1)$-dimensional manifold, we say that a hyperbolic attractor $\Lambda$ is of \emph{codimension 1}   if $\dim E^u(x) = n$ for each $x\in \Lambda$.  
The topology of transitive, expanding, hyperbolic attractors of codimension 1 has been studied extensively, primarily in a series of papers by Plykin.  See for example \cite{MR2095625}, \cite{MR0467836}, 
\cite{Plykin:HAOD}, \cite{Plykin:HAOD-NO}, and \cite{MR771099}. 
These papers outline constraints on the topology of the basins of such 
 attractors, and on the diffeomorphisms admitting them. These results suggest Theorems \ref{thm:main} and \ref{thm:main2} should be true.  However, we use very little of this structure to establish our result other than the basic observation that a transitive, expanding, hyperbolic    attractor of codimension 1 is locally the product of $\R^n$ and a Cantor set.

Given a invariant set $X$ and  $x\in X$ we define 
  \[\omega(x) :=\bigcap_{n\in \N}\overline{\bigcup_{k = n}^\infty f^k(x)}\]
and $$\alpha(x) := \bigcap_{n\in \N}\overline{\bigcup_{k = n}^\infty f^{-k}(x)}$$ called the \emph{$\omega$-limit set} and \emph{$\alpha$-limit set}.  
 If $X$ is compact then $\omega(x)$, and $\alpha(x)$ are nonempty compact invariant sets contained in $\NW(f)$.  
We will be primarily interested in situations where $\alpha(x)$ and $\omega(x)$ are hyperbolic, in which case we will invoke  Lemma \ref{lem:FinAsym} and Proposition \ref{cor:PerDense} below.

 \begin{lem}\label{lem:FinAsym}
{Suppose $\omega(x)$ is a finite hyperbolic set.  Then $x\in \stab{\omega(x)}$.  Similarly, if $\alpha(x)$  is a finite hyperbolic set then $x\in \unst{\alpha(x)}$.}
\end{lem}
\begin{proof}We need only prove the first statement.  Assume the contrary.  If $\omega(x)$ is finite, then there is some $k$, such that every point of $\omega(x)$ is fixed by $f^k$.  For each $y\in \omega(x)$ pick a precompact open neighborhood $U_y\subset B(y,\epsilon)$, where $\epsilon$ is the radius of the local stable manifolds for the dynamics of $f^k$, such that for $z,y\in \omega(x)$, $f^k(U_y) \cap U_z \neq \emptyset$ implies $z = y$.  

Let $y\in \omega_{f^k}(x)\subset \omega (x)$, where $\omega_{f^k}(x)$ is the $\omega$-limit set of $x$ under the dynamics of $f^k$.  If $x\notin W^s(y)$ then there exists an infinite subsequence of $\{f^{nk}(x)\}_{n\in \N}$ disjoint from  $U_y$, whereas $y\in \omega_{f^k}(x)$ implies there exists an infinite subsequence of $\{f^{nk}(x)\}_{n\in \N}$ completely contained in $U_y$.  
Thus we may conclude that there is an infinite subsequence  $\{f^{n_j k}(x)\}\subset \{f^{nk}(x)\}_{n\in \N}$ such that $f^{(n_j-1)k}(x) \in U_y$ and $ f^{n_j k}(x) \notin U_y$ for all $j\in \N$.  Since $f^k(U_y)\sm U_y$ is precompact and $\omega_{f^k}(x)$ is disjoint from $\overline{f^k(U_y)\sm U_y}$, the sequence $\{f^{n_j k}(x)\}$ contains an accumulation point in $\overline{f^k(U_y)\sm U_y}$  disjoint from  $\omega_{f^k}(x)$ contradicting the definition of $\omega_{f^k}(x)$.  
\end{proof}

Recall that a compact hyperbolic set $\Lambda$ is called \emph{locally maximal} if there exists an open set $\Lambda\subset V $ such that $\Lambda = \bigcap_{n\in \Z} f^n(V).$  
The following result is a basic corollary of the Anosov Closing Lemma (see e.g. \cite{Katok:1995p8357}).
 \begin{prop}\label{cor:PerDense}
Let $\Omega$ be a compact, nonwandering, hyperbolic set for a diffeomorphism $f\colon M\to M$ and let $\Per(f)$ be the set of periodic points for $f$.  Then $\Omega\subset \overline{\Per(f)}$.  In particular if $\Lambda$ is a locally maximal compact hyperbolic set, then the periodic points $\Per(\restrict f \Lambda)$ are dense in $\NW(\restrict f \Lambda)$.
\end{prop}

The following Lemma will be of use in the proofs of Theorems \ref{thm:main} and \ref{thm:LM} when considering case were $\dim E^\sigma (x) = 0$ for some $\sigma \in \{s,u\}$.
{\begin{lem}\label{lem:trivialsplitting}
Let $\Lambda\subset M$ be a  compact hyperbolic set for a diffeomorphism $g\colon M\to M$. The  points $y \in \Lambda$ with the property that $E^\sigma(y) = \{0\}$ for some $\sigma \in \{s,u\}$ are periodic and isolated in $\Lambda$.  In particular there are no such points when $\Lambda$ is perfect.  
\end{lem}}
\begin{proof}
Let $y$ be such a point. 
By passing to the inverse if necessary, we may assume that $\dim E^u(y) = n$ (where $\dim M = n$).  Then $\omega(y) \subset \Lambda$ contains a periodic point $p$ such that $\dim E^u(p) = n$.  Pass to an iterate  so that  $p$ is fixed by $g$.  
For any $0<\epsilon<\epsilon_0$, 
where $\epsilon_0$ is the radius of local unstable manifolds for $g$, there is some $k$ such that $g^k(y) \in \locUnst p$.  However, the definition of $\locUnst p$ implies thus that $y = g^{-k} (g^k(y)) \in \locUnst p$ for every $\epsilon>0$, 
hence $p = y$.  

To see that such points are isolated let $p\in \Lambda$ be such that $\dim E^u(p) = n$.  Then $p$ is periodic by the above.   Assume that for any $\epsilon>0$  the set $\locUnst p \cap \Lambda$ contains a point $x$ distinct from $p$.  By taking $\epsilon$ sufficiently small we may assume that $\dim E^u(x) = n$.  Clearly any such $x$ can not be periodic as $\{g^{-k}(x)\}$ contains points arbitrarily close to $p$, hence an infinite number of points.  This  contradicts that $x$ must be periodic by the above result.   
\end{proof}

 
For compact  hyperbolic sets, local maximality is equivalent to the existence of a 
\emph{local product structure} \cite{Katok:1995p8357}.  We recall the following definition.
\begin{defn}\label{def:LPS} A  hyperbolic set $\Lambda$ is said to have \emph{  local product structure} if there is some $\epsilon>0$ and $0<\delta< \epsilon$ such that for every $x\in \Lambda$, the map $$\phi\colon (\locUnst [\delta] x \cap \Lambda)  \times (\locStab [\delta] x \cap \Lambda)\to \Lambda$$ given by $$\phi\colon (y,z) \mapsto \locStab y \cap \locUnst z$$ is well defined and  maps its domain homeomorphically onto its image.  
\end{defn}
For any $0< \eta\le \delta$ we may find some open set $V$ such that $V\cap \Lambda = \phi(  (\locUnst  [\eta] x \cap \Lambda)  \times (\locStab [\eta] x \cap \Lambda) )$.  

\begin{defn}\label{def:LPC}
Given any  $\eta\le \delta$ and any open set $V$ such that $$V\cap \Lambda = \phi(  (\locUnst  [\eta] x \cap \Lambda)  \times (\locStab [\eta] x \cap \Lambda) )$$we call the pair $(V,\eta)$ a \emph{local product chart} of radius $\eta$ centered at $x$.  
\end{defn}
Recall that given a compact hyperbolic set $\Lambda$, we may always find  $\delta>0$ and $\epsilon>0$ such that $\locUnst x \cap \locStab y$ is a singleton whenever $x,y\in \Lambda$ are such that $d(x,y)< \delta$.  Thus, a local product structure simply asserts that this intersection is contained in $\Lambda$. 

For a hyperbolic set exhibiting a local product structure we define a \emph{canonical isomorphism} (also called a $u$-projection, or $u$-holonomy) between subsets of stable  manifolds.  

\begin{defn}[Canonical Isomorphism]\label{def:CI}  Let $(V,\eta)$ be a local product chart centered at $x$.  Let $x'\in \locUnst  [\eta] x$, and let $D\subset \locStab  [\eta] x$, $D'\subset \locStab {x'}$.  Then $D$ and $D'$ are said to be \emph{canonically isomorphic} 
if $y\in D\cap \Lambda$  implies $D'\cap \locUnst y \neq \emptyset$ and $y'\in D'\cap \Lambda$ implies $D \cap \locUnst {y'} \neq \emptyset$. 

 We similarly define a canonical isomorphism between subsets of local unstable manifolds.  
\end{defn}

 Given a locally maximal, compact hyperbolic set  $\Lambda$, there are various partitions of the set of nonwandering points known as the \emph{spectral decomposition} of $\Lambda$ (see e.g. \cite{Katok:1995p8357}, \cite{Smale:DDS}).  For our purposes, the spectral decomposition asserts that  there exist an $n>0$ and a partition  $\NW(\restrict f \Lambda) =\Omega_1\cup \dots\cup \Omega_l$ into a finite number of disjoint closed sets $\Omega_k$, such that each $\Omega_k$ is a  compact topologically  mixing set for $f^n$.  The spectral decomposition implies the following.

\begin{prop}\label{cor:topologically mixingAttractor}
Let $\Lambda$ be a compact, hyperbolic attractor, possibly containing wandering points.  Then $\Lambda$ contains a topologically mixing hyperbolic attractor for some iterate of $f$.
\end{prop}
\begin{proof}
Note that $\Lambda$ is locally maximal.  If $B$ denotes the basin of $\Lambda$ then $\NW(\restrict f B) \subset \Lambda$, hence $\restrict f B$ satisfies Smale's Axiom A. Let $\{\Omega_1, \dots , \Omega_l\}$ be the spectral decomposition of $\NW(\restrict f \Lambda)= \NW (\restrict f B)$, and let $n$ be such that each $\Omega_j$ is topologically mixing for $f^n$.  We define a relation on the $\Omega_k$ by $\Omega_j\gg\Omega_k$ if {$W^u(\Omega_j )\cap W^s(\Omega_k) \neq \emptyset$}. 
Note that $\Lambda$ contains full unstable manifolds, hence  if $x,y \in \Lambda$, then $W^u(x)\cap W^s(y)\subset \Lambda$ which  by hyperbolicity implies the intersection $W^u(x)\cap W^s(y)$ is transverse if nonempty.  Thus $f$ restricted to the basin of $\Lambda$ satisfies Smale's Axiom B  (see \cite{Smale:DDS}, 6.4).  By Proposition 8.5  of \cite{Smale:DDS},  $\gg$ is  a partial ordering.

If some $\Omega_j$ is not an attractor for $f^n$ then its unstable saturation {$W^u(\Omega_j)\subset \Lambda$} contains wandering points for $f$ (hence also for $f^n$).  Let $x\in W^u(\Omega_j)$ have a wandering neighborhood $U$ (that is $f^k(U) \cap U = \emptyset$ for all $k\in \Z\sm\{0\}$) and let $\epsilon>0$ be such that $W^u_\epsilon(x) \subset U$.  Now $\omega(x)\subset \NW(\restrict f\Lambda)$, thus there is some $k$ and $m$ such that $\locUnst{f^{ m n}(x)}\cap W^s(\Omega_k) \neq \emptyset$ and hence $W^u_\epsilon(x) \cap \stab{\Omega_k}\neq \emptyset$. Note that the continuity of the local unstable manifolds  
and uniform hyperbolicity imply that $$\unst{\Omega_k}\subset \overline{\bigcup_{m>0} f^m(\locUnst{ x})}.$$ Thus for $x$ to be a wandering point we must have  $k\neq j$.  Indeed otherwise we would have \[x\in \unst{\Omega_j}\subset \overline{\bigcup_{m>0} f^m(\locUnst{ x})} \subset \overline{\bigcup_{m>0} f^m(U)}\] from which we see that $U\cap f^m (U) \neq \emptyset$ for some $m$.

 Hence if $\Omega_j$ is not a mixing attractor for $f^n$ then $\Omega_j\gg \Omega_k$ for some $k\neq j$.  Since $\gg$ is a partial ordering of a finite set, we may find a maximal element $\Omega_{M}$ in  $\{\Omega_1, \dots , \Omega_l\}$.  (Recall that an element $z \in S$ of a partially ordered set $(S, \le )$ is called a \emph{maximal element} if there is no $w\in S\sm \{z\}$ such that $z\le w$.)  
Any such a $\Omega_{M}$ is thus a topologically mixing attractor for $f^n$. \end{proof}

We will deal with the intersection of two $C^1$ submanifolds whose dimensions may not be complementary.  Given two subspaces $U,W \subset T_x (M)$ we will say that $U$ and $W$ are   in \emph{general position}  if $\dim( U+ W) = \min\{m, \dim U + \dim W\}$ where $m= \dim M$.  We then say that two smooth submanifolds $S, N$  of an $m$-dimensional manifold intersect \emph{transversally at $x\in S\cap N$} if $T_xS $ and $ T_xN$ are in general position.  

We will need the Lefschetz Fixed Point Theorem in our argument (see \cite{Katok:1995p8357} for more discussion).  Recall that a continuous map on a compact manifold $f\colon M\to M$ induces a linear map $f_{*,k}\colon H_k(M, \Q)\to H_k(M, \Q)$ and that the \emph{Lefschetz number} of $f$ is defined to be $$L(f) := \sum_{k = 0}^{\dim M} (-1)^k \tr f_{*,k}.$$
Given an isolated fixed point $p$, suppose that the unit sphere $S$ centered at $p$ (in some local Euclidean neighborhood of $p$) isolates $p$ from all other fixed points.  Then the \emph{fixed point index} of $p$, denoted $I_f(p)$, is defined to be the degree of the map $$\dfrac{\Id - f}{\|\Id - f\|}\colon S \to S.$$
The Lefschetz Fixed Point Theorem ensures that for a map $f\colon M\to M$ containing only isolated fixed points we have $$L(f) = \sum _{\{x \mid f(x) = x\}} I_f(x).$$


\section{Proof of  Theorem \ref{thm:main} and Theorem \ref{thm:main2}.}

In what follows we assume that $\Lambda$ is a compact, topologically mixing,  expanding hyperbolic attractor of  codimension 1   for some 
 diffeomorphism $f\colon M \to M$ where $M$ is an $(n+1)$-dimensional manifold and $\dim E^u(x) = n$ for all $x\in \Lambda$.  Note that this implies $\Lambda$ is connected.  Consider a second 
 diffeomorphism $g\colon M\to M$ for which  $\Lambda$ is a hyperbolic set. 
Note that we make no assumptions on the dimensions of the hyperbolic splitting for $g$, nor do we assume that either the stable or the unstable manifold for $g$ is tangent to the leaves of $\Lambda$.  

Unless necessary we will suppress mention of $f$ in what follows.  For $\sigma \in \{s,u \}$ and $x\in \Lambda$ denote by $E^\sigma (x)$ the invariant $\sigma$-subspaces for $g$, and by $W^\sigma_\epsilon( x)$, and $W^\sigma (x)$ the local and global $\sigma$-manifolds for $g$.  When necessary we will write $W^\sigma_f(x)$ for the corresponding global manifolds for $f$.  For $x\in \Lambda$ denote by $\Lf(x)$  the path-connected component or \emph{leaf} of $\Lambda$ containing $x$.  Note that for all $x\in \Lambda$, $ \Lf(x)= W^u_f(x)$ is the unstable manifold for the dynamics of $f$, hence  is homeomorphic to $\R^n$ and dense in $\Lambda$.  
We will adopt the metric adapted to the dynamics of  $g$ on $\Lambda$ and let $\epsilon _0$ be the radius of the local stable and unstable manifolds in the adapted metric.  

Note that by Lemma \ref{lem:trivialsplitting} we may assume that both the hyperbolic distributions are nontrivial under the new dynamics $g$.  That is $\dim E^\sigma (x) \ge 1$ for all $x\in \Lambda$ and $\sigma\in \{s,u\}$.  

To  prove Theorem \ref{thm:main} we will first show that $\Lambda$ exhibits a local product structure under the dynamics of  $g$.  In particular, we will show that only one of the invariant distributions $E^\sigma(x)$ is in general position with respect to  $T_x \Lambda$.  (By $T_x \Lambda$ we mean the $n$-dimensional subspace of $T_xM $ tangent to the leaf of $\Lambda$ through $x$.) This will establish that $\Lambda$ is either an attractor or a repeller.  We will then show that each leaf of $\Lambda$ is either expanding or contracting, and thus obtain the dimension of the hyperbolic splitting and as a consequence derive  that $\Lambda$ is topologically mixing for $g$.  


\subsection{Uniformity of the Transverse Dynamics.}
We first establish uniform behavior of $g$ transverse to $\Lambda$.  
\begin{prop}\label{thm:UnfTrans}
For $\Lambda$ and $g$ as in the hypothesis of Theorem \ref{thm:main}, 
either $E^u(x) \subset T_x\Lambda$ for all $x\in \Lambda$, or   $E^s(x) \subset T_x\Lambda$ for all $x\in \Lambda$.  
\end{prop}

Since $\Lambda$ is connected and the distributions $ \restrict {E^\sigma} \Lambda$ are continuous, Proposition \ref{thm:UnfTrans} is equivalent to the statement that there are no points $x\in \Lambda$ such that both subspaces $E^s(x)$ and $E^u(x)$ are in general position with respect to $T_x\Lambda$.
To prove Proposition \ref{thm:UnfTrans}, we assume such a point exists and derive a contradiction.  First, we examine the limit sets of any such point.  

\begin{lem} \label{lem:alphaOmega} 
If $E^s(x)$  is in general position with respect to $T_x\Lambda$, then for $y \in \omega (x)$ we have $\unst y \subset \Lambda$.  Similarly, if  $E^u(x)$ is in general position with respect to $T_x\Lambda$, then for  $ y'\in \alpha (x)$ we have  $\stab {y'} \subset \Lambda$.
\end{lem}

\begin{proof} 
We prove only the first conclusion as the result for the stable manifolds  is obtained by passing to the inverse.
For any  
$0<\epsilon<\epsilon_0$ and $\xi \in \Lambda$ define $$\delta _\epsilon (\xi):= \sup _{z\in W^u_\epsilon (\xi)}  \inf \{d(z,y )\mid {y\in \Lambda}\}. $$ By continuity of local unstable manifolds, $\delta_\epsilon(\xi)$ is continuous on $\Lambda$.  Note that if $W^u_\epsilon(x)$ is not contained in a leaf of   $\Lambda$ then  it contains a point not contained in $\Lambda$, hence  for any such point $\delta_\epsilon(x)> 0$.

If $E^s(x)$ is nontrivial and in general position with respect to $T_x\Lambda$ then  $\locStab x $ is a nontrivial manifold transverse to $\Lf(x)$ at $x$.  For any $0<\epsilon'<\epsilon_0$, the transversality of $\locStab x$ and $\Lf(x)$ and the fact that leaves of $\Lambda$ are  codimension-1 immersed submanifolds  guarantees that $U = \bigcup _{y\in \Lambda} W^s_{\epsilon'} (y) $ contains an open neighborhood of  $x$.  Thus there is some $0< \epsilon< \epsilon_0$ such that $W^u_\epsilon (x) \subset U$.  For this $\epsilon$ and $x$  we trivially obtain the upper bound
$$\delta_\epsilon (x)\le \sup _{z\in W^u_\epsilon (x)}  \inf \left\lbrace d(z,y ) \mid y\in \Lambda \ \mathrm{and} \  z\in W^s_{\epsilon'}(y)\right\rbrace  \le \epsilon'.$$ 
Then,
$$W^u_\epsilon (g(x)) \subset g(W^u_\epsilon (x)) \subset g\Big( \bigcup _{y \in\Lambda} W^s_{\epsilon'} (y)\Big)\subset 
\bigcup _{y \in \Lambda } W^s_{\lambda \epsilon'} (y)$$
and $\delta_\epsilon (g(x))\le \lambda \epsilon'$.  

Recursively, we have $W^u_\epsilon (g^n(x))\subset \bigcup _{y \in \Lambda} W^s_{\lambda^n {\epsilon'}} (y)$, and 
hence $\delta_\epsilon (g^n(x)) \le \lambda ^n{\epsilon'}$.  By continuity, for any $y \in \omega (x)$, we have $\delta_\epsilon(y) = 0$ and hence $\locUnst y \subset \Lambda$.  But since the choice of $\epsilon$ is uniform over all $y\in \omega(x)$ we must have $\unst y = \bigcup_{n\in \N} f^n(\locUnst{f^{-n}(y)}) \subset \Lambda$ for $y\in \omega (x)$.  
\end{proof}

Note that for $x$ as in Lemma \ref{lem:alphaOmega},  $\Lambda^a(x) = \overline{\bigcup_{y\in \omega( x)} \unst y} \subset \Lambda$ is a nontrivial hyperbolic attractor for $g$ (as it contains full unstable manifolds) and likewise $\Lambda^r(x) = \overline{\bigcup_{y\in \alpha (x)} \stab y} \subset \Lambda$ is a nontrivial hyperbolic repeller for $g$.  
As both are trivially locally maximal, Proposition \ref{cor:PerDense} guarantees that $\Lambda^r(x)$ and $\Lambda^a(x)$ contain periodic points, say $q$ and $p$ respectively.

\begin{rem}\label{rem:PP}
For $x$ as in Lemma \ref{lem:alphaOmega}, there exists  periodic points $q\in \Lambda^r(x),p\in\Lambda^a(x)$
such that $E^s(q)\subset T_q\Lambda$ and $E^u(p)\subset T_p\Lambda$.
\end{rem}

 We know (see e.g.  \cite{MR771099}) that $\Lambda$ is locally the Cartesian product of $\R^n$ and a Cantor set.   We make the following related definitions involving the local structure of $\Lambda$.
\begin{defn}
 A \emph{local $\Lambda$-chart} is any connected open set $V\subset M$ such that $V\cap \Lambda$ is homeomorphic to the product of $\R^n$ and a Cantor set.
\end{defn}

\begin{defn}
If $V$ is a local $\Lambda$-chart  then  for any $x\in \Lambda\cap V$ we will call  the connected component of $V\cap \Lambda$ containing $x$ the \emph{local leaf through $x$}, denoted by $\Lf_V(x)$,
\end{defn}

Note that if $V$ is a local $\Lambda$-chart, then for $x\in \Lambda \cap V$ the set $V\sm \Lf_V(x)$ contains two components.  Also note that for $x\in \Lambda \cap V$ and $y \in \Lambda \cap V\sm \Lf_V(x)$ the set $V\sm (\Lf_V(x)\cup \Lf_V(y))$ contains three components, only one of which contains both $\Lf_V(x)$ and $ \Lf_V(y)$ in its boundary.

\begin{defn}\label{def:between}
Let $V$ be a local $\Lambda$-chart and let $x\in V\cap \Lambda$, $y \in \Lambda \cap V\sm \Lf_V(x)$.  We say $z$ is \emph{between $x$ and $y$} if $z$ is contained in the unique open set in $V\sm (\Lf_V(x)\cup \Lf_V(y))$  containing both $\Lf_V(x)$ and $ \Lf_V(y)$ in its boundary.  
\end{defn}
\newcommand{\cant}{\mathfrak{C}}

Let $\cant\subset [0,1]$ be the middle-third  Cantor  set.  Given a local $\Lambda$-chart we may find a map $\phi\colon V\cap \Lambda \to \R^n \times \cant$ which is a homeomorphism onto its image and preserves a natural ordering on the local leaves $\{\Lf_V(y)\}$.  That is, if $z$ is between $x$ and $y$  and $\pi$ is the projection to the second coordinate $\pi\colon \R^n \times \cant \to \cant\subset [0,1]$ then either 
\[\pi(\phi(x))< \pi(\phi(z))< \pi(\phi(y)) \quad \mathrm{or} \quad \pi(\phi(y))< \pi(\phi(z))< \pi(\phi(x)).\]

\begin{defn}
Let $V$ be a local $\Lambda$-chart and let $\phi, \pi$ be as above.  Let $\gamma\colon(a,b)\to V$ be a continuous curve  and let $\mathcal{T} = \{t\in(a,b)\mid \gamma(t) \in \Lambda\}$.  Then we say $\gamma$ is \emph{monotonic in $V$} if the map
\[\pi\circ \phi\circ \gamma\colon \mathcal{T}\subset(a,b) \to [0,1]\]
is (non-strictly) monotonic.  
\end{defn}

\begin{defn}
Two continuous curves  $\gamma_1$ and $\gamma_2$ will be called \emph{canonically isomorphic}, if there is some local $\Lambda$-chart $V$ such that $\gamma_1\subset V$, $\gamma_2 \subset V$, both $\gamma_j$ are monotonic in $V$, and $\gamma_1$ and $\gamma_2$ intersect the same local leaves of $V\cap \Lambda$. 
\end{defn}

\begin{lem}\label{lem:chart}
Let $V$ be a local $\Lambda$-chart.  Then $g(V)$  is  also a local $\Lambda$-chart. 
In particular, if $\gamma_1, \gamma_2\subset V$ are canonically isomorphic  curves, monotonic in $V$, then $g(\gamma_1)$ and $g(\gamma_2)$ are canonically isomorphic curves monotonic in $g(V)$.\end{lem}
\begin{proof}
If $V\cap \Lambda$ is homeomorphic to $\R^n\times C$ via a homeomorphism $\phi$ then $g(V) $ is homeomorphic to $\R^n\times C$  via the homeomorphism $\phi \circ g\inv$.  Additionally $\gamma_1, \gamma_2\subset V$ are canonically isomorphic if $\gamma_1\cap \Lf_V (x) \neq \emptyset \iff \gamma_2\cap \Lf_V (x) \neq \emptyset$.  But as $g$ preserves path-connected components of $V\cap \Lambda$ we must have $g(\gamma_1)\cap \Lf_{g(V)} (g(x)) \neq \emptyset \iff g(\gamma_2)\cap \Lf_{g(V)} (g(x))\neq \emptyset $.  \end{proof}

Now, since $\Lambda$ is a topologically mixing attractor for $f$, the leaves of $\Lambda$  
are dense in $\Lambda$. In particular, taking the periodic points from Remark \ref{rem:PP} 
for a fixed $\epsilon<\epsilon_0$ we have $\Lf(q) \cap  W^s_\epsilon(p) \neq \emptyset$.  Consequently, we may  choose a local $\Lambda$-chart $V$ containing $q$, such that $V\cap W^s_\epsilon(p)$ is a nonempty open subset of $W^s_\epsilon(p)$. 
\begin{rem}\label{rem:Curves}
We may find some 
curve $\gamma_q \subset W_\epsilon^u(q)$ containing $q$ and a corresponding canonically isomorphic 
curve  $\gamma_p \subset V\cap W^s_\epsilon(p)$, both of which are 
monotonic in $V$,  and intersect more than one local leaf. \end{rem}
Note that $q\in \gamma_q $ by construction, but that we do not \emph{a priori} have $p\in \gamma_p$.  

Passing to an iterate, we may assume that $p,q$ are fixed by $g$.  
We see intuitively that the transverse structure of $\Lambda$ contracts near $p$, and expands near $q$ under iterates of $g$ which should contradict the canonical isomorphism.  To derive a precise contradiction, we introduce a  measure transverse to the lamination of $\Lambda$, given by a canonical local disintegration of the measure of maximal entropy (for the original map $f\colon \Lambda \to \Lambda$)  into a product measure.  An explicit construction of the measure of maximal entropy   and its disintegration is given in \cite{Sinai:1968p7692} for uniformly hyperbolic transitive diffeomorphisms of compact manifolds, and in \cite{Ruelle:1975p7690}  for basic sets of  Axiom A diffeomorphisms.  

Recall that given an Axiom A diffeomorphism $f$ (respectively a locally maximal compact hyperbolic set $\Lambda$) and a spectral decomposition $\Omega_1\cup \dots\cup \Omega_l = \NW(f)$ (respectively $\Omega_1\cup \dots\cup \Omega_l = \NW(\restrict f \Lambda)$) the sets $\bigcup _{j\in \Z} f^j (\Omega_i)$ are called the \emph{basic sets} of $f$ (respectively of $\restrict f \Lambda$).  That is, a basic set is a compact, transitive, 
open subset of the nonwandering points of $f$ (respectively $\restrict f \Lambda$).  

Clearly, mixing hyperbolic attractors are basic sets.   
We present the theorem as stated in \cite{Ruelle:1975p7690} (adapted to our notation), the conclusion of which applies directly to $\Lambda$.

\begin{thm}[Ruelle, Sullivan  \cite{Ruelle:1975p7690}]\label{thm:RSS}
Let $\Omega$ be a basic set. Let $h $ be the topological entropy of $\restrict f \Omega$. Then there is an $\epsilon>0$ so that for each $x\in \Omega$ there is a measure $\mu_x^u$ on $\locUnst x$ and a measure $\mu_x^s$ on $\locStab x$ such that: 
\begin{enumerate}
\item $\supp  \mu_x^u= \locUnst x\cap \Omega$ and $\supp  \mu_x^s= \locStab x\cap \Omega$.
\item $\mu^u $ and $\mu^s$ are invariant under canonical isomorphism (see Definition \ref{def:CI}).  That is, if $x'\in \locStab [\eta] x$ and $D\subset \locUnst [\eta]  x, D'\subset \locUnst  [\eta] {x'}$ are canonically isomorphic then $\mu^u_x(D )= \mu^u_{x'} (D')$,  and if  $x'\in \locUnst [\eta] x$ and $D\subset \locStab  [\eta] x, D'\subset \locStab  [\eta]{ x'}$ are canonically isomorphic then $\mu^s_x (D )= \mu^s_{x'} (D')$.
\item $f_* \mu ^u_x = e^{-h} \mu ^u_{f(x)} $ on $\locUnst{f(x)}$ and  $f\inv _* \mu ^s_x = e^{-h}  \mu ^s_{f\inv(x)} $ on $\locStab{f\inv(x)}$.
\item The product measure $\mu^u_x \times \mu^s_x$ is locally equal to Bowen's measure of maximal entropy.  
\end{enumerate}
\end{thm}

	We define a measure $\nu$, defined  on all $C^1$ curves in $M$, as follows. 
Given such a curve  $\gamma\colon [0,1]\to M$ we may partition $[0,1]$ by a countable collection of $\{t_j\}$ such that $\gamma ([t_i, t_{i+1}])$ is either disjoint from $\Lambda$ or monotonic in a local $\Lambda$-chart $V_i$.  
In the  case that $\gamma ([t_i, t_{i+1}])$ is disjoint from $\Lambda$ define $\nu(\gamma ([t_i, t_{i+1}])) = 0$.  
In the case where $\gamma ([t_i, t_{i+1}])$ is monotonic in some local $\Lambda$-chart $V$, we fix an $x\in V$ and define $\pi_x\colon \gamma ([t_i, t_{i+1}])\cap (\Lambda\cap V)\to W^s_{V,f}(x)$ by $\pi_x(z) = \Lf_V(z) \cap W^s_{V,f}(x)$, where $ W^s_{V,f}(x)$ is the connected component of  $W^s_{f}(x)\cap V$ containing $x$, and define $\nu(\gamma ([t_i, t_{i+1}]) ) = \mu^s_x(\pi_x(\gamma ([t_i, t_{i+1}]) \cap \Lambda))$.
We then extend $\nu$ to all of $\gamma$ by additivity.
Note that the assumption that the curves $\gamma ([t_i, t_{i+1}])$ are monotonic in $V_i$ and the fact that $\mu^s$ is invariant under canonical isomorphism ensure that $\nu$ is well defined.
	
	By the proof of Theorem \ref{thm:RSS} in the appendix of \cite{Ruelle:1975p7690}, it is a direct computation to show that $\nu$ is non-atomic and $\nu(\gamma)>0$ for any curve intersecting more than one leaf of $\Lambda$.  Hence, for $\gamma_p, \gamma_q$ as in Remark \ref{rem:Curves}, we have $0<\nu(\gamma_p ) = \nu(\gamma_q)<\infty$.

\begin{proof}[Proof of Proposition \ref{thm:UnfTrans}.]
Let $x$ be a point such that $E^s(x) \not \subset T_x\Lambda$ and $E^u(x) \not \subset T_x\Lambda$.   Then Lemma \ref{lem:alphaOmega}, and Remarks \ref{rem:PP} and \ref{rem:Curves} apply.  Take the periodic points $p$ and $q$ as in Remark \ref{rem:PP}, and pass to an iterate so that  they are fixed by $g$.  Let $V, \gamma_p, \gamma_q$ be as in Remark \ref{rem:Curves}.


To arrive at a contradiction, we may reduce $\epsilon$ in Remark 3.10 so that 
$\gamma_q\subset W^u_\epsilon(q) $ intersects every connected component of $W^u_{\epsilon}(q)\cap \Lambda$.
 Since $W^u_\epsilon(q) \subset g(W^u_\epsilon(q))$ we must have that 
$g(\gamma_q)$ contains a subset canonically isomorphic to $\gamma_q$.  Thus the sequence $\{\nu (g^k(\gamma_q))\}_{k\ge 0}$ is bounded away from zero.

On the other hand we may find some 
curve $\tilde \gamma\colon [0,1]\to W^s_{\epsilon_0}(p)$, monotonic in some local $\Lambda$-chart, such that $p \in \tilde \gamma((0,1))$.  Since $\gamma_p\subset W^s_{\epsilon_0}(p)$, for some $n$ we have that $g^n(\gamma_p)$ is canonically isomorphic to a subset of $\tilde \gamma $.  Furthermore, for any open set $U\subset \tilde \gamma$ containing $p$, we can find some $N>0$ such that $g^N(\tilde \gamma)$ is canonically isomorphic to a subset of $U$.  But since $$\bigcap _{U\subset \tilde \gamma:p\in U} U= \{p\}$$ we must have a subsequence $ \{n_j\}$ such that  $\nu (g^{n_j}(\gamma_p)) \to\nu(\{p\}) = 0$
contradicting the fact that  $\nu (g^{n}(\gamma_p) )= \nu( g^{n}(\gamma_q))$  is bounded away from $0$ for all $n\ge 0$. 

Thus no such point $x$ exists, and by continuity of the hyperbolic splitting, and connectedness of $\Lambda$, we conclude that either  $E^u(x) \subset T_x\Lambda$ for all $x\in \Lambda$, or   $E^s(x) \subset T_x\Lambda$ for all $x\in \Lambda$. 
\end{proof}


\subsection{Dimension of the Hyperbolic Splitting.}
Let $\Lambda$ and $g$ be as above.  By  passing to the inverse if necessary and invoking Proposition \ref{thm:UnfTrans} we may assume for all $x\in \Lambda$ that $E^u (x) \subset T_x\Lambda$ and hence $E^s (x)$ is in general position with respect to $T_x\Lambda$.  Note again that Lemma \ref{lem:trivialsplitting} guarantees that $\dim E^\sigma (x)\ge 1$ for all $x\in \Lambda$ and $\sigma \in \{s,u\}$.  
Under these assumptions we prove that $\dim E^s(x) = 1$,  $\dim E^u(x) = n$, and  $\Lambda$ is a topologically mixing hyperbolic attractor for $g$.

For $x\in \Lambda$, denote by $\Lf_\eta(x)$ the path-connected component of $\Lambda\cap B(x,\eta)$ containing $x$.  
We call a path-connected component $\bound \subset \Lambda$ a \emph{boundary leaf} if for every $x\in \bound$ there exists some $\eta>0$ such that $B(x,\eta)\sm \Lf_\eta(x)$ contains two connected components, one of which is disjoint from $\Lambda$.  
Because $\Lambda $ is locally the product of a Cantor set and $\R^n$, boundary leaves exist.  Also, from \cite{MR771099} we know that there are only a finite number of boundary leaves.  

\begin{lem}\label{lem:CmtBdrSet}
Fix a boundary leaf \ $\bound\subset  \Lambda$, and denote by $G_\epsilon$ the set of points in $\bound$ such that one component of $ W^s_\epsilon(y)\sm\Lf_\epsilon( y)$ is disjoint from $\Lambda$.  Then $G_\epsilon$ is an open subset of $\bound$ and $K_\epsilon :=\overline {G_\epsilon}$ is a compact subset of $\bound$, where $\bound$ is endowed with the topology of $\R^n$ as an immersed submanifold.   
\end{lem}

\begin{proof}The openness of $G_\epsilon$ (in $\bound$) follows from the continuity of local stable manifolds.  
We show sequential compactness of $\overline {G_\epsilon}$ in $\bound$.
Given a sequence $\{x_j\}\subset G_\epsilon$, compactness of $\Lambda$ guarantees an \emph{ambiently} convergent subsequence $x_{j_k} \to x^*$.  We show for any ambiently convergent sequence  $\{x_j\} \to x^*$ with $\{x_j\}\subset G_\epsilon$,  that $x^* $ is contained in $ \bound$ and that the convergence $\{x_{j}\}\to x^*$ occurs in the internal topology of $ \bound$.  

Let $\{x_j\} \to x^*$ be such a sequence.  
Let $V$ be a local $\Lambda$-chart  containing $x^*$.  Truncating our sequence we may assume that $\{x_j\}\subset V\cap \Lambda$.  If the sequence $\{x_j\}$ doesn't converge to $x^*$ in $\bound$, then by passing to a subsequence we may assume that 
$x^*\notin \Lf_V(x_{j})$ for any $j$.  

Thus we may choose a subsequence $\{x_{n_j}\}\to x^*$ such that $\Lf_V(x_{n_j}) = \Lf_V(x_{n_i})$ if and only if $ i = j$ and such that if $i<j$ then $x_{n_j}$ is between $x_{n_i}$ and $x^*$ (see Definition \ref{def:between}).  But since $W^s(x_{n_j})$ is everywhere transverse to the leaves of $\Lambda$, and the local leaves $\{\Lf_V(x_{n_i})\}$ 
vary continuously, there is some $J$ such that $W^s_\epsilon (x_{n_J}) \cap \Lf_V(x^*)\neq \emptyset$ and $W^s_\epsilon (x_{n_J}) \cap \Lf_V(x_{n_{J-1}})\neq \emptyset$.  But $\Lf_V(x^*)$ and $ \Lf_V(x_{n_{J-1}})$ are in different components of $ V\sm \Lf_V(x_{n_{J}})$ contradicting that $ x_{n_{J}}\in G_\epsilon$. 
\end{proof}

Now, we know there are only a finite number of boundary leaves, and since the map $g$ preserves boundary leaves, it must permute them.  Thus there is some $k$ such that $g^{k}(\bound ) \subset \bound$ for  every boundary leaf $\bound$.  Fix a boundary leaf $\bound$.  
Then $g^k\colon \Lambda\to \Lambda$ induces a diffeomorphism $\tilde g\colon \bound \to \bound$.  Since we assume that $E^u(x) \subset T_x\Lambda$ for all $x\in \Lambda$, the induced diffeomorphism  $\tilde g$ is uniformly hyperbolic in the induced metric, with a $D\tilde g$-invariant hyperbolic splitting $E^s_{\tilde g}(x) = E^s(x)\cap T_x \Lambda$ and $E^u_{\tilde g}(x) = E^u(x)$.  

We have the following corollary.

\begin{cor}
Fix a boundary leaf $\bound$ and  $0<\epsilon<\epsilon_0$ such that the compact set $K_\epsilon \subset \bound$ as guaranteed by Lemma \ref{lem:CmtBdrSet} is nonempty.  Then  $\tilde g\inv (K_\epsilon)\subset K_\epsilon$, hence there exists a periodic point on any boundary leaf.
\end{cor}
\begin{proof}
Pick any $x\in K_\epsilon$.  For every $n\ge 0$, we know that $W^s_\epsilon(g^{-nk}(x))\subset g^{-nk}(W^s_\epsilon(x))$ and hence one component of $W^s_\epsilon(g^{-nk}(x))\sm \Lf_\epsilon(g^{-nk}(x))$ is disjoint from $\Lambda$.  Hence $g^{-nk}(x) \in K_\epsilon$ for any $n\ge 0$.  
 
Now considering the $\alpha$-limit set of $x$ under the dynamics $\tilde g\colon \bound\to \bound$  
we have that  $\alpha(x) \subset K_\epsilon$.  Since  $\alpha (x)$ is a nonempty compact hyperbolic set  for $\tilde g\colon \bound\to \bound$ with $\alpha (x) \subset \NW(\tilde g)$, 
by Proposition \ref{cor:PerDense} the periodic points of $\tilde g$ accumulate on $\alpha(x)$.  
\end{proof}

Let $p\subset \bound$ be a boundary periodic point.  Denote by  $W^s_{\tilde g}(p)\subset \bound$ the stable manifold of $p$  for the induced hyperbolic diffeomorphism $\tilde g\colon \bound \to \bound$.  We have that $W^s_{\tilde g}(p)$  is the path-connected component of $W^s(p) \cap \bound$ containing $p$.  Choose an $\epsilon>0$ such that $p\in G_\epsilon$.  Denoting by $W^s_{\eta ,\tilde g}(p)$ the connected component of $W^s(p) \cap \bound\cap B(p,\eta)$ containing $p$, we may choose $\eta>0$ such that $W^s_{\eta ,\tilde g}(p) \subset G_\epsilon$ and $W^s_{\eta ,\tilde g}(p) $ is homeomorphic to a $(k-1)$-dimensional disk, where $k = \dim E^s$ along $\Lambda$.  Let $J$ be such that $\tilde g^J(p) = p$.  Then $W^s_{\tilde g} (p) =\bigcup_{n\in \N} \tilde g^{-n J}(W^s_{\eta ,\tilde g}(p))$, whence $W^s_{\tilde g}(p) \subset K_\epsilon$.

We may now assemble our observations above into the following proposition.
\begin{prop}\label{thm:dim}
For $g\colon \Lambda \to \Lambda$ as in Theorem \ref{thm:main}, with $E^s(x)$ in general position with respect to $T_x (\Lambda)$ at some $x\in \Lambda$, we have $\restrict {\dim E^s} \Lambda = 1$. 
\end{prop}
\begin{proof}
By Proposition \ref{thm:UnfTrans} we have that $E^s(x)$ in general position with respect to $T_x (\Lambda)$ for every $x\in \Lambda$.  Assume that $2\le \dim E^s\le n$.  Let $p$ be a periodic point in a boundary leaf $\bound$, pass to an iterate such that $p$ is fixed by $g$, and denote by $\tilde g$ the induced diffeomorphism on $\bound$.  Let $\Gamma =\overline{\bigcup_{x\in W^s_{\tilde g}(p)} \alpha (x)} \subset K_\epsilon$.  Note that $\Gamma $ must be infinite.  Indeed if $\Gamma$ were finite, then by Lemma \ref{lem:FinAsym},
we would have $ W^s_{\tilde g}(p)\subset W^u_{\tilde g}(\Gamma)$.  Furthermore, 
the uniform hyperbolicity of $\restrict  g \Lambda$ would guarantee that $ W^s_{\tilde g}(p)$ and $W^u_{\tilde g} (\Gamma)$ intersect transversally, whence $ W^s_{\tilde g}(p)\cap W^u_{\tilde g}(\Gamma)$ is at most countable.  However if $\dim E^s\ge 2$ then $W^s_{\tilde g}(p)$ contains a continuum.  Thus we conclude that $\Gamma$ is an infinite hyperbolic set for $\tilde g$ contained in $ \NW(\tilde g)$, which by Proposition \ref{cor:PerDense} implies that there are an infinite number of periodic points for $\tilde g$ (hence for $g$) on the boundary leaf $\bound$.

Now $\tilde g\colon \bound \to \bound$ naturally induces a homeomorphism of the one-point compactification $\overline g\colon   S^n\to S^n$ that projects to  $\tilde g$ with  a fixed point at $\infty$.  
\begin{claim}\label{claim:infty}
For some $N$ and all $k\ge N$, $\infty$  is an isolated fixed point for  $\overline g^k$, with fixed point index $+1$.  
\end{claim}
\begin{proof}
Let $0<\eta<\epsilon_0$ be such that $ p \in K_\eta$.  
Let $S$ be the unit sphere at $\infty$ in an appropriately chosen Euclidean chart such that $S$ separates $\infty$ from $K_\eta$.  We note that any fixed point for $\tilde g^k$ must be contained in $K_\eta$, hence $S$ is an isolating sphere for all iterates of $\tilde g$. 
Let $D\subset \bound$ be the disk bounded by $S$.  We may cover $D$ with a finite number of $\{G_{\epsilon_j}\}$.  Let $\delta$ be the minimum $\epsilon_j$ in this cover.  Since $S$ is disjoint from $K_\eta$, we know that for all $x\in S$, both components of $W^s_\eta(x)\sm \Lf_\eta(x)$ intersect $\Lambda$.  We may find an $N$ such that $\lambda^N\eta < \delta$.  But then $\tilde g^k(S)$ is disjoint from $D$ for all $k\ge N$.  Thus, considering $S$ as the unit sphere at $\infty$ we have $|\bar g^k(x)| < |x|$ for all $k\ge N$ and $x\in S$.  Thus $\dfrac{\Id - \bar g^k}{\|\Id -\bar g^k\|}\colon S\to S$ is homotopic to the identity map, hence has degree $+1$.  
\end{proof}

Now the Lefschetz number for any homeomorphism of $S^n$ is either $-2, 0, $ or $2$.  Furthermore, since $\tilde g\colon \bound \to \bound$ is hyperbolic, all fixed points of $\tilde g^n$ are isolated and the index of every fixed point of $\tilde g^n$ in $\bound$ can be computed in terms of the dimension of $E^u$ and whether $\tilde g$ preserves the orientation of the distribution $E^u$.  Thus all fixed points of $\tilde g^n$ have the same index, and thus for any $n>N$ as in Claim \ref{claim:infty} the number of fixed points of $\tilde g^n$ is at most $3$ contradicting the infinitude of periodic points if $\dim E^s\ge 2$.  Hence we must have $\dim E^s = 1$. \end{proof}


\subsection{Proof of the Theorems \ref{thm:main} and \ref{thm:main2}}
\

\begin{proof}[{Proof of Theorem \ref{thm:main}.}]
By passing to the inverse if necessary, from Proposition \ref{thm:UnfTrans} we can conclude for every $x\in \Lambda$ that $E^s(x)$ is in general position with respect to $T_x \Lambda$. 
By Proposition \ref{thm:dim},  $\dim E^s(x) = 1$ for all $x\in \Lambda$, hence applying Lemma \ref{lem:alphaOmega} to a boundary periodic point $p$,  one then has $W^u(p) = \Lf(p)$.  Thus $\overline {W^u(p)} =\overline {\Lf(p)} = \Lambda $ is a hyperbolic attractor for $g$.  By Proposition \ref{cor:topologically mixingAttractor} we know that $\Lambda$ contains a topologically mixing attractor $\Lambda'$ for some iterate of $g$.   But then for any $x\in \Lambda'$, we  have $W^u(x)=\Lf (x)$ is dense in $\Lambda'$, thus 
$\Lambda = \Lambda'$, and $\Lambda$ is a mixing   expanding attractor of codimension 1 for $g$.  
\end{proof}

\begin{proof}[{Proof of Theorem \ref{thm:main2}.}]
If $\Lambda$ is a transitive  expanding hyperbolic   attractor of codimension 1 for $f$, then by Remark \ref{lem:perfecttrans}, $\NW(\restrict f \Lambda ) = \Lambda$.  By the spectral decomposition it decomposes into a finite number of topologically mixing attractors $\{\Omega_j\}$ for some iterate of $f$.  Note that each $\Omega_j$ is connected.  Since $g$ preserves $\Lambda$ it must permute its connected components, hence there is some $k$ such that $g^k$ fixes each $\Omega_j$.  Applying Theorem \ref{thm:main} to $\restrict{g^k} {\Omega_j}$ we conclude that $\Omega_j$ is a topologically mixing  expanding   attractor (or contracting repeller) of codimension 1 for $g^k$.  But then for a fixed $j$ the set $\Lambda_j = \bigcup_i g^i(\Omega_j)$ is a transitive expanding  hyperbolic   attractor (or contracting repeller)  of codimension 1 for $g$.
\end{proof}


\section{Proof of Theorem \ref{thm:LM}.}
\newcommand{\clocUnst}[2][\epsilon]{\check W^u_{#1}(#2)}
\newcommand{\clocStab}[2][\epsilon]{\check W^s_{{#1}}(#2)}
The proof of Theorem \ref{thm:LM} follows  from   Lemma \ref{lem:trivialsplitting}, and  Claim \ref{clm:cantor} and Theorem  \ref{thm:LMTD} below. 
\begin{proof}[{Proof of Theorem \ref{thm:LM}.}]
Let $\Lambda$ be a nonwandering, locally maximal, compact hyperbolic set for a surface diffeomorphism $f\colon S\to S$.  Fix $V$ open so that $\Lambda = \bigcap _{n\in \Z} f^n(V)$.  We note that by Proposition \ref{cor:PerDense} we have $\Lambda \subset \overline{\Per (\restrict f V)}$ hence by local maximality $\Lambda = \overline{\Per (\restrict f \Lambda)}$ and in particular $\Lambda = \NW( \restrict f \Lambda)$.  Thus the spectral decomposition applies and we may assume without loss of generality that $\Lambda$ is topologically mixing.  
If $\dim E^\sigma(x) = {0}$ for some $x\in \Lambda$ and $\sigma\in \{s,u\}$ then Lemma \ref{lem:trivialsplitting} implies that $\Lambda$ is a periodic orbit, hence Theorem 1.4 holds trivially.  

Thus, let us assume that $\dim E^\sigma(x) = 1$ for both $\sigma\in \{s,u\}$.    We exhaust the following 3 cases.  

\begin{case} Suppose  that $\Lambda$ has nonempty interior.  Then by Theorem 1 of \cite{MR2199438}, $\Lambda = S$ and thus Theorem \ref{thm:LM} holds trivially.\end{case}

\begin{case} Next we consider the case when $\Lambda$ has empty interior but contains a topologically embedded curve $\gamma$.  Let $\delta, \epsilon$ be as in the definition of the local product structure for $\Lambda$ under the dynamics of $f$ (see Definition \ref{def:LPS}).  We note that under these hypotheses, either $\gamma \subset W^u(x)$ or $\gamma \subset W^s(x)$ for some $x\in \Lambda$ since otherwise
the set $$\bigcup_{\{z,y\in \gamma\mid d(y,z)<\delta\}} \locUnst y \cap \locStab z\subset \Lambda$$ would have nonempty interior.  Without loss of generality assume $\gamma \subset W^u(x)$ for some $x\in \Lambda$. Topological mixing thus  implies $\unst x\subset \Lambda$ for all $x\in \Lambda$. Since $\Lambda$ has empty interior, this implies $\Lambda$ is a nontrivial mixing hyperbolic attractor, hence the conclusion of Theorem \ref{thm:LM} is true by Theorem \ref{thm:Fisher}.
\end{case}
\begin{case} Finally we  assume that $\Lambda$ has empty interior, and that no curve may be topologically embedded in $\Lambda$. In particular, no curve may be embedded in  $\stab x$ or $\unst x$ for any $x \in \Lambda$.  Thus $\stab x \cap \Lambda$ and $\unst x \cap \Lambda$ are totally disconnected, which by the local product structure on $\Lambda$ implies that  $\Lambda$ is totally disconnected.  

By Claim \ref{clm:cantor} below, either $\Lambda$ is finite, or $\Lambda$ is locally the product of two Cantor sets.  In the former case, the proof of the conclusion to Theorem \ref{thm:LM} is trivial; the conclusion in the latter case follows from Theorem \ref{thm:LMTD} below. \qedhere \end{case} \end{proof}

We proceed with the statement and proof of 
 Claim \ref{clm:cantor}.

\begin{claim}\label{clm:cantor}
Let $\Lambda$ be a compact, topologically mixing, locally maximal, totally disconnected hyperbolic set for a surface diffeomorphism $f\colon M\to M$.  Then either $\Lambda$ is finite, or $\overline{W^\sigma_\epsilon (x) }\cap \Lambda$ is a Cantor set for every $x\in \Lambda$ and $\sigma\in \{s, u \}$.  
\end{claim}

\begin{proof}
Assume that $\Lambda$ is not finite.  By Lemma \ref{lem:trivialsplitting} we may assume that $\dim E^\sigma(x) = 1$ for all $x\in \Lambda$ and $\sigma \in \{s,u\}$.  To establish the claim we show that $\overline{W^\sigma_\epsilon(x)}\cap\Lambda$ is perfect for every $x\in \Lambda$  and $\sigma \in \{s,u\}$.  

Let $p\in \Lambda$ be periodic.   If $\Lambda$ is mixing then it contains a point $x$ whose orbit is both forwards and backwards dense in $\Lambda$.   Let $(V, \epsilon) $ be a local product chart at $p$ and let $\phi$ be the canonical homeomorphism $\phi\colon (\locStab p \cap \Lambda ) \times ( \locUnst p \cap \Lambda) \to V\cap \Lambda$.  There exists a sequence $\{n_j\}$ so that $f^{n_j}(x) \to p$ with $\{f^{n_j}(x) \} \subset V$.  Let $(z_j, w_j) $ be such that $\phi((z_j, w_j) ) = f^{n_j}(x) $.  Then $z_j \to p$ and $w_j \to p$.  Furthermore, we must have $z_{j} \neq p$ for any $j$, since otherwise we would have $f^N(x) \in \locUnst p$ for some $N$, which contradicts that $x$ has a backwards dense orbit.  Similarly  $w_{j} \neq p$ for any $j$.  Thus $p$ is not isolated in $W^\sigma_\epsilon(p) \cap \Lambda$.  

Similarly, let $y\in \Lambda$ be non-periodic.   By Proposition \ref{cor:PerDense} the periodic points of $\restrict f  \Lambda$ are dense in $\Lambda$ hence there is a sequence $\{p_j\}$ of distinct periodic points accumulating on $y$.   As above, let $(V, \epsilon) $ be a local product chart at $y$, let $\phi$ be the canonical homeomorphism, and let $(z_j, w_j) $ be such that $\phi((z_j, w_j) )= p_j$.  Then if $j \neq k$ we must have $z_j \neq z_k$ and $w_j \neq w_j$, since otherwise we would have either $p_k \subset \locUnst{p_j} $ or $p_k \subset \locStab {p_j}$.  But then we clearly can find infinite subsequences of $\{z_j\}$ and $\{ w_j \}$  disjoint from $\{y\}$.  Hence $y $ is not isolated in $W^\sigma_\epsilon (y) \cap \Lambda$.  
\end{proof}

\subsection{Proof of Theorem \ref{thm:LMTD}.}
The proof of Theorem \ref{thm:LM} given above will be complete after proving the following.  

\begin{thm}\label{thm:LMTD}
Let $f\colon S\to S$ be a 
diffeomorphism of a surface $S$, and let $\Lambda$ be a topologically mixing, locally maximal, totally disconnected,  compact hyperbolic set for $f$ containing an infinite number of points. 
Assume $g\colon S\to S$ is a second diffeomorphism such that $\Lambda$ is hyperbolic for $g$.   Then $\Lambda$ is a locally maximal set for $g$.
\end{thm}

To simplify notation in the proof of Theorem \ref{thm:LMTD} we will suppress the dynamics of $f$ and denote by $\Gamma^1_\epsilon(x)$ and $ \Gamma^2_\epsilon(x)$ the local stable and unstable manifolds of $x$ for $f$.  Write $\check \Gamma ^j _\epsilon (x):= \Gamma ^j_\epsilon (x)\cap \Lambda$.   Note that by Lemma \ref{lem:trivialsplitting}, $\Gamma ^j_\epsilon (x)$ is a curve and that  by Claim \ref{clm:cantor}, $\check \Gamma^j_\epsilon(x)$ is a Cantor set.  Thus the notations $\locUnst x$ and $\locStab x$ are reserved for the local stable and unstable manifolds of $x$ under the dynamics of $g$.  Note that by Lemma \ref{lem:trivialsplitting} both $W_\epsilon^\sigma(x)$ are curves.  Write $\check W^\sigma_\epsilon(x) := W^\sigma_\epsilon(x) \cap \Lambda $  for $\sigma\in \{s,u\}$.

\begin{defn}
We say that $(V, \eta)$ is  a \emph{local $f$-product chart} centered at $x$ if  $V\cap \Lambda \cong  \check \Gamma^1_\eta(x)\times \check \Gamma^2_\eta(x)$ via the canonical homeomorphism  $\phi\colon (y,z)\mapsto \Gamma^2_\epsilon(y)\cap  \Gamma^1_\epsilon(z)$.
\end{defn}
That is, a local $f$-product chart is a local product chart at $x$ under the dynamics of $f$ (see Definition \ref{def:LPC}).  Note that we may always choose $V$ in such a way that  $\Gamma^j_\eta(x)$ separates $V$ into two components.   As before, if $V$ is a local $f$-product chart, for any $y \in \Lambda \cap V$ denote by $\Gamma^k_V(y)$ the connected component of $\Gamma^k(y)\cap V$ containing $y$.  

For $j\in \{1,2\}$ we will say that $x\in \Lambda$ has the \emph{$j$-boundary property} if $x$ is an endpoint of an open interval in $\Gamma^j_\eta(x)\sm \check \Gamma^j_\eta(x)$.  Alternatively, $x$ has the $1$-boundary property, if for any local $f$-product chart  $(V,\eta)$ with canonical homeomorphism $\phi\colon \check \Gamma^1_\eta (x) \times  \check  \Gamma^2_\eta(x) \to V\cap \Lambda$, we may find some $\delta>0$ and $U\subset V$ so that $U\cap \Lambda = \phi (\check \Gamma^1_\delta (x) \times  \check  \Gamma^{2}_\eta(x))$ and $\ \Gamma^{2}_\eta (x) $ separates $U$ into two components, one of which is disjoint from $\Lambda$.  A  similar equivalent definition holds for points with the $2$-boundary property.  We say that a point has the \emph{boundary property} if it has the $j$-boundary property for both $j\in \{1,2\}$.  Note that by Claim \ref{clm:cantor} such points are dense in $\Lambda$.  

We now show that the dynamics of $g$ preserves the product structure of $\Lambda$ for $f$.  
 
\begin{lem}\label{lem:4.1}
Under the hypothesis of Theorem \ref{thm:LMTD}, there are $\eta, \eta'>0$ so  that for all $x\in \Lambda$ and $k\in \{1,2\}$ there is a $j\in \{1,2\}$ such that  $g(\check \Gamma^k_\eta(x) ) \subset \check \Gamma^j_{\eta'} (g(x))$.
\end{lem}

\begin{proof} Compactness of $\Lambda$ guarantees that we may find  $\eta>0$ and $\eta'>0$ so  that for each $x\in \Lambda$ there exists open $V(x)$ and $V'(x)$ such that  $(V(x), \eta)$ is a local $f$-product chart centered at  $x$, $(V'(x), \eta')$ is a local $f$-product chart centered at $g(x)$, $g(V(x)) \subset V'(x)$, and $V(x) \sm \Gamma^k_\eta(x)$ has two components.  Fix a $k\in \{1,2\}$.  Since the families of curves $\{\Gamma^1_{\eta'}(x)\} $ and $\{\Gamma^2_{\eta'}(x)\}$ are transverse at each $x\in \Lambda$, and since $\Gamma^1_{\eta'}(g(x)) $, $\Gamma^2_{\eta'}(g(x))$, and $g(\Gamma^k_{\eta}(x)) $ are smoothly embedded curves, we may further reduce $\eta$ and $\eta'$ so that 
there exists a function $\tau\colon\Lambda \to \{1,2\}$ such that for every $x\in \Lambda$ the curve $g(\Gamma^k_\eta (x))$ intersects each curve in the family $\{\Gamma^{\tau(x)}_{V'(x)}(y)\} _{y \in V'(x)\cap \Lambda}$ in at most one point, and does so transversally. We fix this $\eta$  and $\eta'$ to be those guaranteed  in the lemma.

By the continuity of the family of local manifolds $\{\Gamma_\eta^k(y)\}_{y\in \Lambda}$ and of the map $g$, if  the conclusion to the lemma fails at a point $x\in \Lambda$, then it fails for  our fixed $k$ and $\eta$ at all points in some neighborhood of $x$ in $\Lambda$.  Hence, by the density of points with the boundary property, it is enough to check that the lemma holds at points with the boundary property.  

  Let $x$ be a point with the boundary property and assume the conclusion of the lemma is false.  Without loss of generality (by relabeling the $\Gamma^j$) suppose $\tau(x) = 2$.
Let $\phi$ denote the canonical  homeomorphism from $\check \Gamma^1_{\eta'} (g(x)) \times\check \Gamma^2 _{\eta'}(g(x))$ to $\Lambda \cap V' $ given by the local product structure for $\Lambda$ under $f$ (see Definition \ref{def:LPS}).

Let $\mathfrak{C}\subset [0,1]$ denote the middle-third Cantor set, and let $\psi_j$ be any homeomorphism between $\check \Gamma^j_{\eta'} (g(x))$ and a subset of $\mathfrak{C}$.  Then $\Xi= \psi_1 \times \psi_2$ is homeomorphism between $\phi\inv (V')$ and a subset of $\mathfrak{C}\times \mathfrak{C}$.  Let $\pi_j\colon \Xi(\phi\inv(V'))\to [0,1]$ be the coordinate projections.  Then the fact that $g(\Gamma^k_\eta (x))$ intersects each curve in $\{\Gamma^2_{V'}(y)\} _{y \in V'\cap \Lambda}$ in at most one point, implies that the function 
\[\pi_1 \circ \Xi\circ \phi \inv\colon g(\Gamma^k_\eta (x))\cap \Lambda \to [0,1]\]
is injective. 

If  $g(\check \Gamma^k_\eta(x) )= g(\Gamma^k_\eta(x) ) \cap \Lambda $ is not contained in   $\Gamma^1_{\eta'}(g(x)) $ 
then we may find some open $U\subset  g(\Gamma^k_\eta(x) ) $ such that  $U\cap \Lambda\neq \emptyset$ and 
$U$ intersects each curve in the family  $\{\Gamma^1_{V'}(y)\} _{y \in V'\cap \Lambda}$ in at most one point and does so transversally.  Then the function
\[\pi_2 \circ \Xi\circ \phi \inv\colon U \cap \Lambda \to [0,1]\]
is injective.  
Since each $\pi_j\circ \Xi\circ \phi\inv$ is injective on $U$, there are only countably many points in $\Xi(\phi\inv(U\cap \Lambda ))$ such that one coordinate is an endpoint of an open interval in $[0,1]\sm\mathfrak{C}$.  Hence we conclude that $U$ contains a point $\phi (\xi, \zeta)$ such that neither $\xi$ nor $\zeta$ is an endpoint of an open interval in $\Gamma^1_{\eta'}(g(x))\sm \check \Gamma^1_{\eta'}(g(x))$ or  $\Gamma^2_{\eta'}(g(x))\sm \check \Gamma^2_{\eta'}(g(x))$ respectively.
  Thus there are points of $\Lambda$ arbitrarily close to $g(\Gamma^k_\eta(x)  )$ in either component of $g(V)\sm g(\Gamma^k_\eta(x) )$.  But that implies there are points of $\Lambda$ arbitrarily close to $\Gamma^k_\eta(x)$ contained in either component of $V\sm \Gamma^k_\eta(x) $ contradicting that $x$ has the $\hat k$-boundary property where $\{\hat k\} = \{1,2\}\sm \{k\}$.
\end{proof}

\begin{cor}\label{cor:tang}
For each $x\in \Lambda$ there is a permutation $\tau_x$ of the set $\{1,2\}$ such that $D_xg (T_x \Gamma^k(x))  \subset T_{g(x)} \Gamma^{\tau_x(k)}(g(x)) $.
\end{cor}

\begin{proof}Suppose that $D_x g(T_x \Gamma^k(x)) \cap T_{g(x)} \Gamma^{j}(g(x))  = \{0\}$ for both $j\in \{1,2\}$.  Then $g(\Gamma^k(x))$ is transverse to both $\Gamma^{j}(g(x))$ at $g(x)$, hence taking $\eta$ small enough we would have $g(\Gamma_\eta ^k(x)) \cap \Gamma_\eta^{j}(g(x)) = \{g(x)\}$ for both $j$, contradicting Lemma \ref{lem:4.1} and the fact that $\check \Gamma _\eta ^k(x)$ is a Cantor set.  
\end{proof}

\begin{cor}\label{cor:tang2}
There is a function $\tau^\sigma\colon \Lambda \to \{1,2\}$ so that for each $x\in \Lambda$ we have $E^\sigma(x) = T_{x} \Gamma^{\tau^\sigma(x)}(x) $. 
\end{cor}
\begin{proof}We prove the result for $\sigma = s$.  For two subspaces $U,V\in T_x M$ let $\measuredangle (U,V)$ denote the angle between them.  
Suppose the conclusion fails at $x\in \Lambda$.  If $T_{x} \Gamma^{k}(x) \not \subset E^s(x)$ for both $k\in \{1,2\}$ then by hyperbolicity, \[\measuredangle \left(D_x g^n(T_x\Gamma^{k}(x)), E^u(g^n(x))\right) \to 0\quad \mathrm{as} \ n\to \infty \] for both $k\in \{1,2\}$ which implies that 
 \[\measuredangle \left(D_xg^n(T_x \Gamma^{1}(x)), D_xg^n(T_x \Gamma^{2}(x))\right) \to 0\quad \mathrm{as} \ n\to \infty \] 
 contradicting the fact that \[\min\left\{\measuredangle \left(T_{y} \Gamma^{1}(y), T_{y} \Gamma^{2}(y)\right) \mid y \in \Lambda\right\}>0\] by compactness of $\Lambda$.
\end{proof}

Note that the functions $\tau^\sigma$ are locally constant on $\Lambda$.  Hence up to  locally relabeling the manifolds $\Gamma_\epsilon^k(x)$ we may assume for every $x\in \Lambda $ that $E^s (x) \subset T_x \Gamma^1 (x)$ and $E^u (x) \subset T_x \Gamma^2 (x)$.

\begin{prop}\label{lem:cont}
Under the hypothesis of Theorem \ref{thm:LMTD}, and the relabeling above, there is an $\epsilon>0$ so that for every point  $x\in \Lambda$  we have
$\check \Gamma^{1}_\epsilon(x)\subset \check W^s_\epsilon(x)$, and 
$ \check \Gamma^{2}_\epsilon(x)\subset \check W^u_\epsilon(x)$.  
\end{prop}

\begin{proof}  We only prove the statement for $\Gamma^1$. Endow $M$ with the  metric adapted to the dynamics of $g$ on $\Lambda$ and let $\epsilon_0$ be the radius of local stable and unstable manifolds for $g$ in the adapted metric.  
By compactness of $\Lambda$ we may choose positive $\delta, \eta, \eta'$ with the properties that:
 \begin{itemize}
{  \item[---] $\delta< \eta \le \epsilon_0$,
 \item[---]  for each $\xi \in \Lambda$ there exist open $V$ and $V'$ such that  $(V, 2\eta)$ is a local $f$-product chart centered at  $\xi$, $(V', 2\eta')$ is a local $f$-product chart centered at $g(\xi)$, $g(V) \subset V'$, and  for all $k\in \{1,2\}$ there is a $j\in \{1,2\}$ such that  $g(\check \Gamma^k_{2\eta}(\xi) ) \subset \check \Gamma^j_{2\eta'} (g(\xi))$ as guaranteed by Lemma \ref{lem:4.1},
\item[---]  for all $x\in \Lambda$ we have  $\locUnst[2\eta] x \cap \Gamma^{1}_{2\eta} (x) = \{x\}$,
\item[---] if $x,y\in \Lambda$ are such that $d(x,y)\le \delta$ then $\locStab [\eta] x \cap \locUnst [\eta] y$ is a singleton. } \end{itemize}
  Fix this $\eta$, $\eta'$ and $\delta$ in what follows.

We make the following definitions:
\begin{align*}
	\mu_0 &:= \sup_{x\in \Lambda} \sup_{y\in\locUnst[\eta]x \sm \{x\}} \{d(g(x),g(y))/d(x,y)\}\\
  A(x) &:= \overline {B(x, \eta)} \sm B(x, \mu_0 \inv \eta)\\
	r(x) &:= \inf \left\{d \left(z, \Gamma^{1}_{2\eta} (x)\right)\mid {z\in \locUnst[\eta]x \cap A(x)}\right\}\\
	\rho &:= \min\{r(x)\mid{x\in \Lambda}\}.
\end{align*}
By the continuity of the local manifolds $r\colon \Lambda \to \R$ is continuous.  The assumption that   $\locUnst[2\eta] x\cap \Gamma^{1}_{2\eta} (x) = \{x\}$ ensures $r(x) \neq 0$ for any $x\in \Lambda$, whence $\rho>0.$

Now, let $x$ be a point such that the conclusion $\check \Gamma^{1}_\epsilon(x)\subset \check W^s_\epsilon(x)$ fails for every $\epsilon>0$.  We may find a  $y_1 \in \check \Gamma^{1}_{\delta}(x)$ such that $\locStab[\eta]  {y_1} \cap   \locUnst[\eta] x= \{w_1\}$ for some $w_1 \neq x$.    Indeed, otherwise we would have $\locStab[\eta] y \cap \locUnst[\eta]x = \{x\}$ for all $y \in \check \Gamma^{1}_{\delta} (x)$ implying $\check  \Gamma^{1}_{\delta} (x)\subset \check W^s_\eta(x)$.

Let $0<\delta_2<\delta$ be such that $y_1 \notin \Gamma^{1}_{\delta_2} (x)$.  Then as above we may find a $y_2 \in \check \Gamma^{1}_{\delta_2} (x)$ so that $\locStab[\eta] {y_2} \cap   \locUnst[\eta]x = \{w_2\} \neq \{x\}$.  Recursively, we find a sequence $\{y_i \}\subset \check \Gamma^{1} _{\delta} (x)$ such that $y_i \to x$ and $\locStab[\eta]{y_i} \cap   \locUnst[\eta]x = \{w_i\} \neq \{x\}$ for all $i\in \N$.

Since  $w_i\to x$, for any $n\in \N$ we can find a $j$ such that $g^k(w_j ) \in B(g^k(x), \mu_0\inv\eta)$ for all $k\le n.$  
Furthermore, for every $j\in \N$ there is some $n_j$ such that \begin{align*}
g^k(w_j) &\in{B(g^k(x), \mu_0\inv\eta)} \quad \mathrm{ for} \ k<n_j \\g^{n_j}(w_j)& \notin{B(g^{n_j}(x), \mu_0\inv\eta)}.\end{align*}  That is, $k = n_j$ is the smallest $k$ such that $g^{k}(w_j)\in A(g^{k}(x))$.  

Now, for every $y_j$ above and $k\le n_j$ we must have \[g^k(y_j) \in \check \Gamma ^{1} _{2\eta} (g^k(x)).\]  Indeed, the conclusion is true if $k = 0$ (recall $\delta<\eta$). Hence inductively assume that $g^k(y_j) \in \check \Gamma ^{1} _{2\eta}( g^k(x))$ for $0\le k<n_j - 1$.  Then by Lemma \ref{lem:4.1} \[g(g^k(y_j))\in  \Gamma ^{1} _{2\eta'} (g^{k+1}(x)).\] On the other hand, \[d(g^{k+1}(x),g^{k+1}(y_j)) \le d(g^{k+1}(x),g^{k+1}(w_j))+d(g^{k+1}(w_j),g^{k+1}(y_j)) \le  \eta + \lambda^{k+1}\eta< 2\eta\] hence \[g(g^k(y_j))\in  \Gamma ^{1} _{2\eta} (g^{k+1}(x)).\]

Now, we may find an $n_j$ such that $\lambda^{n_j}\eta< \rho$.  Let $z = g^{n_j}(x).$  Then we have 
\begin{align*}g^{n_j}(w_j) &\in  \locUnst[\eta]{z} \cap A(z)\\
g^{n_j} (y_j) &\in \check \Gamma ^{1} _{2\eta}( z)\end{align*}
whence  \[r(g^{n_j}(x)) \le   d(g^{n_j} (w_j),g^{n_j} (y_j) ) \le \lambda^{n_j}\eta< \rho\] contradicting the definition of $\rho$.  Thus at each point $x\in \Lambda$ there is some  $\epsilon$, possibly depending on $x$, such that the conclusion holds.  By compactness of $\Lambda$ we may find a uniform $\epsilon>0$ for which the lemma holds.  
\end{proof}

We note that this establishes Corollary \ref{cor:algined}.  Finally we may prove Theorem \ref{thm:LMTD}.
\begin{proof}[{Proof of Theorem \ref{thm:LMTD}.}]
  Since $\Lambda$ is a compact hyperbolic set for $g$, we know that there is some $\delta>0$ and some $\epsilon>0$ such that if $d(x,y)\le \delta$ then $ W^u_\epsilon(x) \cap W^s_\epsilon(y)$ is a singleton  for all $x,y\in \Lambda$.  Thus we need only to establish that  $W^u_\epsilon(x) \cap W^s_\epsilon(y) \subset \Lambda$. 

By Proposition \ref{lem:cont} and by the continuity of the local manifolds, we may reduce $\delta$ and $\epsilon$ so that for all $x\in \Lambda$ and $y\in B(x, \delta)\cap \Lambda$ the inclusions  
$\check \Gamma^{1}_{\epsilon}(y)\subset\clocStab[\epsilon]y $ and $\check \Gamma^{2}_{\epsilon}(y)\subset\clocUnst[\epsilon]y $ hold. Further decreasing $\delta$ and $\epsilon$, the local product structure of $\Lambda$  for $f$ ensures that for 
$y \in B(x,\delta)\cap \Lambda$,  the set $\Gamma^{2}_{\epsilon}(x)\cap \Gamma^{1}_{\epsilon} (y)\subset \Lambda$ is a singleton.

But then  $\locUnst x\cap \locStab y =  \Gamma^{2}_{\epsilon}(x)\cap \Gamma^{1}_{\epsilon} (y) \in \Lambda$.  Hence we may find uniform $\delta>0$  and $\epsilon>0$ so that $\Lambda$ has a local product structure under $g$, and  is thus  locally maximal for $g$.  
\end{proof}

\section{Acknowledgments.}
The author would like to extend his gratitude to Todd Fisher for suggesting this problem and for his useful comments.  Additionally the author would like to thank his PhD adviser, Boris Hasselblatt, for his many suggestions and insights regarding this paper.

\def\cprime{$'$}


\begin{thebibliography}{15}

\bibitem{MR2199438}
T.~Fisher.
\newblock {Hyperbolic sets with nonempty interior}.
\newblock {\em Discrete Contin. Dyn. Syst.} {\bf 15}(2) (2006), 433--446.

\bibitem{Fisher:2006p1846}
T.~Fisher.
\newblock {The topology of hyperbolic attractors on compact surfaces}.
\newblock {\em Ergodic Theory Dynam. Systems} {\bf 26} (2006), 1511--1520.

\bibitem{MR0271990}
J.~Franks.
\newblock {Anosov diffeomorphisms}.
\newblock In {\em Global {A}nalysis ({P}roc. {S}ympos. {P}ure {M}ath., {V}ol.
  {XIV})}. Amer. Math. Soc., Providence, R.I., 1970, pp. 61--93.

\bibitem{MR2095625}
V.~Grines and E.~Zhuzhoma.
\newblock {On structurally stable diffeomorphisms with codimension one
  expanding attractors}.
\newblock {\em Trans. Amer. Math. Soc.} {\bf 357} (2005), 617--667.

\bibitem{MR0271991}
M.~W. Hirsch and C.~Pugh.
\newblock {Stable manifolds and hyperbolic sets}.
\newblock In {\em Global {A}nalysis ({P}roc. {S}ympos. {P}ure {M}ath., {V}ol.
  {XIV})}. Amer. Math. Soc., Providence, R.I., 1970, pp. 133--163.

\bibitem{MR0006493}
W.~Hurewicz and H.~Wallman.
\newblock {\em Dimension {T}heory}.
\newblock Princeton Mathematical Series, v. 4. Princeton University Press,
  Princeton, N. J., 1941.

\bibitem{Katok:1995p8357}
A.~Katok and B.~Hasselblatt.
\newblock {\em Introduction to the Modern Theory of Dynamical Systems}.
\newblock Cambridge University Press, Cambridge, 1995.

\bibitem{MR1632177}
A.~Katok and R.~J.~Spatzier.\newblock{Differential rigidity of {A}nosov actions of higher rank
              abelian groups and algebraic lattice actions}.
 \newblock{\em Tr. Mat. Inst. Steklova} {\bf{216}} (1997), {292--319}.

\bibitem{MR0467836}
H.~Kollmer.
\newblock {On hyperbolic attractors of codimension one}.
\newblock In {\em Geometry and Topology (Lecture Notes in Mathematics, 597)}.
  Springer, Berlin, 1977, pp. 330--334.

\bibitem{MR0277004}
S.~E. Newhouse.
\newblock {On codimension one {A}nosov diffeomorphisms}.
\newblock {\em Amer. J. Math.} {\bf 92}(3) (1970), 761--770.

\bibitem{Plykin:HAOD}
R.~V. Plykin.
\newblock {Hyperbolic attractors of diffeomorphisms}.
\newblock {\em Russian Math. Surveys} {\bf 35}(3) (1980), 109--121.

\bibitem{Plykin:HAOD-NO}
R. V. Plykin.
\newblock {Hyperbolic attractors of diffeomorphisms (the nonorientable case)}.
\newblock {\em Russian Math. Surveys} {\bf 35}(4) (1980), 186--187.

\bibitem{MR771099}
R.~V. Plykin.
\newblock {The geometry of hyperbolic attractors of smooth cascades}.
\newblock {\em Russian Math. Surveys} {\bf 39}(6) (1984), 85--131.

\bibitem{Ruelle:1975p7690}
D.~Ruelle and D.~Sullivan.
\newblock {Currents, flows and diffeomorphisms}.
\newblock {\em Topology} {\bf 14} (1975), 319--327.

\bibitem{Sinai:1968p7692}
J.~G. Sina{\u\i}.
\newblock {Markov partitions and {C}-diffeomorphisms}.
\newblock {\em Funct. Anal. Appl.} {\bf 2} (1968), 61--82.

\bibitem{Smale:DDS}
S.~Smale.
\newblock {Differentiable dynamical systems}.
\newblock {\em Bull. Amer. Math. Soc.} {\bf 73} (1967), 747--817.

\end{thebibliography}
\end{document}